\documentclass[10pt,a4paper]{amsart}
\usepackage{amsxtra,amscd, amsbsy,eucal,color}
\usepackage{setspace}
\usepackage{amssymb,amsfonts,amsmath,amsthm,amscd,epsfig,mathrsfs,textcomp,fontenc, stackrel}

\newtheorem{theorem}{Theorem}[section]

\newtheorem{corollary}[theorem]{Corollary}
\newtheorem{proposition}[theorem]{Proposition}

\theoremstyle{definition}
\newtheorem{definition}[theorem]{Definition}
\newtheorem{example}[theorem]{Example}

\theoremstyle{remark}
\newtheorem{remark}[theorem]{Remark}

\newcommand{\F}{\Bbbk}
\newcommand{\G}{\mathcal{G}}

\newcommand{\vet}{\varepsilon_t}

\newcommand{\qd}{\nolinebreak\hfill$\Box$\par}
\newcommand{\Hom}{\operatorname{Hom}}
\newcommand{\kernel}{\operatorname{Ker}}
\newcommand{\imagem}{\operatorname{Im}}

\begin{document}

\title[Inner Actions of Weak Hopf algebras]{Inner Actions of Weak Hopf algebras}

\author[D. Bagio]{Dirceu Bagio}
\address{Departamento de Matem\'atica, Universidade Federal de Santa Maria, Brazil}
\email{bagio@smail.ufsm.br}
\author[D. Fl\^ores]{Daiana Fl\^ores}
\address{Departamento de Matem\'atica, Universidade Federal de Santa Maria, Brazil}
\email{flores@ufsm.br}
\author[A. Sant'Ana]{Alveri Sant'Ana}
\address{Instituto de Matem\'atica, Universidade Federal do Rio Grande do Sul, Brazil}
\email{alveri@mat.ufrgs.br}

\thanks{\noindent 2000 \emph{Mathematics Subject Classification: 16S40; 16T05.\newline}}


\begin{abstract}
Let $R$ be an associative ring and $e,f$ idempotent elements of $R$. In this paper we introduce
the notion of $(e,f)$-invertibility for an element of $R$ and use it to define inner actions of weak Hopf algebras.
Given a  weak Hopf algebra $H$ and an algebra $A$, we present sufficient conditions for $A$ to admit an inner action of $H$.
We also prove that if $A$ is a left $H$-module algebra then $H$ acts innerly on the smash product $A\#H$ if and only if $H$ is a quantum commutative weak Hopf algebra.
\end{abstract}

\maketitle


\section*{Introduction}
Weak Hopf algebras were introduced in the literature, in a purely
algebraic approach, by G. B\"{o}hm, F. Nill and K. Szlach\'anyi. Roughly speaking, every weak Hopf algebra has the structure of an algebra and of a coalgebra, but the compatibility between these structures are weakened. Ordinary Hopf algebras and groupoid algebras are standard examples of weak Hopf algebras. Several authors have been publishing results in the setting of weak Hopf algebras generalizing known results for Hopf algebras.

Inner actions of Hopf algebras on algebras were introduced by M. Sweedler in \cite{Sw}. Among other things, the existence of inner actions was
used by Sweedler to obtain results in the setting of the cohomology of Hopf algebras. Later,
R. Blattner, M. Cohen and
S. Montgomery presented in \cite{BCM} a new approach to this subject. Those authors established the foundations of inner actions theory of arbitrary
Hopf algebras. Moreover, inner actions were used to study crossed products. This subject has attracted the attention of several researchers since \cite{BCM}.

The main purpose of this paper is to extend the notion of inner actions given in \cite{BCM} to the setting of weak Hopf algebras. In her PhD thesis \cite{F}, the second author has studied actions of groupoid algebras on the correspondent smash product by conjugation (inner actions of groupoids). With the purpose to generalize these results to the setting of weak Hopf algebras we needed to introduce the concept of inner actions of weak Hopf algebras on an arbitrary algebra.

Let $H$ be a weak Hopf algebra. The natural generalization of conjugation in $H$ by elements of $H$ is the action induced by the adjoint map $ad_h:H\rightarrow H$, defined by $ad_h(g) = h_1gS(h_2)$, for any $h, g \in H$. It is well known that the $\F$-linear map $\cdot:H\otimes H \to H$ given by $h\cdot g = ad_h(g)$, where $h,g \in H$, does not induce a structure of left $H$-module algebra on $H$ (see \cite{AVRC}). In fact, it is true that this map defines a left $H$-module algebra on $H$ if and only if $H$ is quantum commutative (see Theorem 2.15 of \cite{AVRC}).

For an arbitrary left $H$-module algebra $A$ we present in Theorem \ref{analogousLemma1.4forWHA} sufficient conditions that to ensure the existence of inner actions of $H$ on $A$. Also, we prove that if $A$ is a left $H$-module algebra, then $H$ acts innerly on the smash product $A\#H$ if and only if $H$ is quantum commutative(see Proposition \ref{some_equivalences}).

The paper is organized as follows. In Section 1, we recall well-known results related to weak Hopf algebras. The notion of $(e,f)$-invertibility for
any associative ring $R$, where $e,f\in R$ are idempotent elements, is studied in Section 2. Also, we discuss the relation among this notion and the following concepts: normalized pseudo-inverse, Moore-Penrose inverse, Drazin inverse and corner-inner automorphism.
We introduce inner actions of weak Hopf algebras in Section 3, and the main purpose of this section is to investigate under what conditions there exists a inner action of $H$ on an algebra $A$.
In Section 4, we investigate the relationship between adjoint actions and quantum commutativity for weak Hopf algebras under our insight. Particulary, we prove that if $A$ is a left $H$-module algebra then $H$ acts by adjoint action on the smash product $A\#H$ if and only if $H$ is a quantum commutative weak Hopf algebra.

Throughout this paper, $\F$ denotes a field and any $\F$-algebra is assumed to be associative and unital. By $C(A)$ we denote the center of a
$\F$-algebra $A$. Given $S\subset A$ we denote by $C_{A\otimes A}(S)$ the subset of $S$-central elements of $A\otimes A$, that is, $C_{A\otimes A}(S)=\{y\in A\otimes A\,:\,(s\otimes 1)y=y(1\otimes s), \text{ for all } s\in S\}$.
For any $\F$-coalgebra $C$ we denote the coproduct by  Sweedler's notation: $\Delta(c) = c_1 \otimes c_2$, $c\in C$ (summation understood).


\section{Preliminaries}

We recall some classical definitions and results about weak Hopf algebras
which we will use; see, {\em e. g.}, \cite{BNS}, \cite{NikVai} and \cite{NSW} for details and proofs.

\begin{definition}
A {\em weak Hopf algebra} over $\F$ is a six-tuple $(H, m, \mu,\Delta,\varepsilon,S)$ where $H$
is a finite dimensional $\F$-vector space such that $(H, m,
\mu)$ is an algebra, $(H, \Delta, \varepsilon)$ is a coalgebra,
$S:H\to H$ is a $\F$-linear map and the following conditions hold: for any $g,h,l\in H$,
\begin{align}\label{weak}
&\Delta(hg) = \Delta (h) \Delta(g),& \\ \label{1.1}
&\Delta^2(1)=(\Delta(1) \otimes 1)(1\otimes \Delta(1))=(1 \otimes
\Delta(1)) (\Delta(1) \otimes 1),&\\\label{1.2}
&\varepsilon(hgl)=\varepsilon (hg_1) \varepsilon (g_{2}
l)=\varepsilon (hg_2)\varepsilon (g_{1} l),&\\\label{1.5}
& h_1 S(h_2)=\varepsilon (1_1 h) 1_2, &\\ \label{1.6}
&S(h_1) h_2= 1_1 \varepsilon (h 1_2),& \\  \label{1.7}
&S(h_1) h_2 S(h_3)= S(h).&
\end{align}
\end{definition}

The $\F$-linear map $S$ is called the {\em antipode} of $H$. The antipode is the unique invertible anti-algebra and anti-coalgebra homomorphism satisfying the conditions \rm{(4)}-\rm{(6)}.

The {\em target} and the {\em source
counital maps} are given by \eqref{1.5} and \eqref{1.6} and denoted by $\varepsilon_t$ and $\varepsilon_s$,
respectively. Thus,
\begin{align*}
&\varepsilon_{t}(h)=(\varepsilon \otimes id)(\Delta(1)(h \otimes
1))\quad\,\,\,\, \text{and}& &\varepsilon_{s}(h)=(id \otimes
\varepsilon)((1 \otimes h)(\Delta(1)).&
\end{align*}
It is clear that $\varepsilon_{t}$ and $\varepsilon_{s}$ are
$\F$-linear idempotents maps. Furthermore, $S \circ
\varepsilon_t = \varepsilon_s \circ S$ and $S \circ \varepsilon_s
= \varepsilon_t \circ S$.
The {\em target} and the {\em source counital subalgebras} of $H$ are defined
respectively by
\begin{align*}
& H_{t} = \{ h \in H : \varepsilon_{t} (h) = h\}\quad
\text{and} \quad H_{s} = \{ h \in H : \varepsilon_{s} (h) = h\}.&
\end{align*}

\begin{proposition}{\label{used}}
Let $H$ be a weak Hopf algebra. For all $x_s \in H_s$, $x_t \in H_t$ and $g,h\in H$, the following statements hold:
\begin{enumerate}\renewcommand{\theenumi}{\roman{enumi}}   \renewcommand{\labelenumi}{(\theenumi)}
\item $\Delta(1) = 1_1 \otimes 1_2 \in H_s \otimes H_t$,
\item $\Delta(x_s) = 1_{1} \otimes x_s 1_2=1_{1} \otimes 1_2 x_s,\quad \Delta(x_t) = 1_{1} x_{t} \otimes 1_2= x_{t}1_{1}  \otimes 1_2$,
\item $\varepsilon_{s}(\varepsilon_{s}(h) g)=\varepsilon_{s}(hg)$,\quad $\varepsilon_{s}(h) g = g_1 \varepsilon(hg_2)$,\quad
$\varepsilon_s(h \varepsilon_s(g))= \varepsilon_s(h) \varepsilon_s(g)$,
\item $\varepsilon_t(h \vet(g)) = \varepsilon_t(hg)$,\quad $h\varepsilon_t(g)= \varepsilon(h_1g)h_2$,\quad
$\varepsilon_t(\varepsilon_t(h) g)= \varepsilon_t(h) \varepsilon_t(g)$. \qd
\end{enumerate}
\end{proposition}

Let $H$ be a weak Hopf algebra and $A$ a $\F$-algebra. Suppose that $A$ is a left $H$-module and denote by $h \cdot x$ the action of $h\in H$ on $x\in A$. Then, $A$ is called a {\em left $H$-module algebra} if
\begin{align*}
&h \cdot xy = (h_1 \cdot x)(h_2 \cdot y),& &h \cdot 1_{A}=\varepsilon_{t}(h) \cdot 1_A,& &\text{ for all } x,y \in A,\,\,h \in H.&
\end{align*}

In this case, $A$ is a right $H_t$-module with action given by
\begin{align}\label{actiondimitri}
& x \cdot z: = S^{-1}(z) \cdot x = x (z \cdot 1),& &\text{ for all } x \in A,\,\,z \in H_t.&
\end{align}
Since $H$ is a left $H_t$-module via multiplication, we can consider the {\it smash product algebra} $A \# H$, that is, the $\F$-vector space $A \otimes_{H_t} H$ with product
\begin{align}\label{smash}
&(x \# h)(y \# g)=x(h_1 \cdot y) \# h_2 g,& &x,y\in A,\quad g,h\in H.&
\end{align}
The smash product $A\#H$ has a unity given by $1_{A} \# 1_{H}$. Also, we have that
\begin{align}\label{impor}
(x \# 1)(y \# 1)=(x(1_1\cdot y))\cdot 1_2\#1
\stackrel{\eqref{actiondimitri}}{=}x(1_1\cdot y)(1_2\cdot 1)\#1=xy\#1,
\end{align}
for all $x,y\in A$.


\section{$(e,f)$-invertibility}

With the purpose to present the concept of inner actions of weak Hopf algebras will be necessary to consider a kind of weak invertibility in a specific convolution algebra. In this section we present such a notion for arbitrary associative rings, and we relate it to other well known concepts which already have appeared in the literature.

\begin{definition} \label{definvert}
Let $R$ be a ring and $e, f$ nonzero idempotent elements in $R$. An element $u\in R$ is  said {\it $(e,f)$-invertible} if there exists $v\in R$ such that:
\begin{enumerate}\renewcommand{\theenumi}{\roman{enumi}}   \renewcommand{\labelenumi}{(\theenumi)}
\item $uv = e$ and $vu = f$;
\item $u = uf$ and $fv = v$.
\end{enumerate}
In this case, we say that $v$ is {\it an $(e,f)$-inverse} of $u$.
\end{definition}
Note that if $v$ is an $(e,f)$-inverse of $u$ in $R$, then it follows from the definition that $ve = v$ and
$eu = u$. Suppose now that $v$ and $v^{\prime}$ are $(e,f)$-inverses of $u$. Then,
$$v^{\prime} = fv^{\prime} = (vu)v^{\prime} = v(uv^{\prime}) = ve = v.$$
Thus, the $(e,f)$-inverse of $u$ is unique, if it exists.
\smallbreak\smallbreak

\noindent {\it Notation:} Let $u$ be an $(e,f)$-invertible element in $R$. The unique $(e,f)$-inverse  element will be
denoted by $u^{-1}$.
\smallbreak\smallbreak

\begin{remark}
Of course, if $R$ is a ring with unity $1_R$ and we take $e = 1_R = f$ in Definition \ref{definvert}, then the $(1_R,1_R)$-invertible elements coincide with the invertible elements in the ordinary sense. Hence, $(e,f)$-invertibility is a generalization of invertibility. Moreover, any idempotent element $e\in R$ is $(e,e)$-invertible and $e^{-1}=e$.
\end{remark}

Observe that we could have defined $(e,f)$-invertibility without the last condition in Definition \ref{definvert}, but in this case we would lose the uniqueness. Moreover, in this hypothetic situation, if $v$ is an $(e,f)$-inverse of $u$ in $R$, then
\[(eu)(ve) = eee = e,\,\,\, (ve)(eu) = v(uv)(uv)u = fff = f,\,\,\, f(ve) = ve,\,\,\, (eu)f = eu.\]
Thus, if $u, v, e, f \in R$ are as in Definition \ref{definvert}(i), then $ve$ is an $(e,f)$-inverse of $eu$ and the elements $eu$ and $ve$ are the unique elements which satisfy the condition (ii) in Definition \ref{definvert}. So, changing $u$ and $v$ by $eu$ and $ve$ respectively, we can insert the condition (ii) in the Definition \ref{definvert} above.
\smallbreak \smallbreak

\noindent {\it Normalized pseudo-inverse.} We recall that an element $r$ in a ring $R$ is called {\it regular} if there exists $s\in R$ such that $r = rsr$. In this case, the element $s$ is called a {\it generalized inverse} (or sometimes {\it pseudo-inverse}) of $r$. It is clear that in this situation we have that $e := rs$ and $f := sr$ are idempotent elements of $R$ and $r = rf (= er) $. Moreover, we say that a generalized inverse $s$ of $r$ is {\it normalized} if $s$ is a regular element of $R$ with generalized inverse equal to $r$, that is, $r = rsr$ and $s = srs$. In this case, we have that $rs = e$, $sr = f$, $e^2 = e$, $f^2 = f$, $r = rf (=er)$ and $s = fs (=se)$. Thus,  if $r$ has a normalized generalized inverse $s$, then $r$ is $(rs, sr)$-invertible and $r^{-1}=s$.

Conversely, if $e,f$ are idempotent elements in $R$ and $u\in R$ is an $(e,f)$-invertible element, then it is easy to check that
\[uu^{-1}u = uf =u,\quad u^{-1}uu^{-1} = fu^{-1} = u^{-1},\] so that $u$ is a regular element with normalized generalized inverse given by $u^{-1}$.
\smallbreak \smallbreak

\noindent {\it Moore-Penrose inverse.} An involution in a ring $R$ is an additive  map $\ast:R\rightarrow R$, $r\mapsto r^{\ast}$, such that $(rs)^{\ast} = s^{\ast}r^{\ast}$ and $(r^{\ast})^{\ast} =r$. Suppose that $r\in R$ is a regular element with normalized generalized inverse $s$. Then it is possible to ask if the idempotent elements $e = rs$ and $f = sr$ are {\it self-adjoint}, that is, $e^{\ast} = e$ and $f^{\ast} = f$. If this happens, then we say that $s$ is the {\it Moore-Penrose inverse} of $r$; see \cite{P} for more details. Thus, in particular, Moore-Penrose invertible elements in an algebra with involution are $(e,f)$-invertible elements. For example, in \cite{HM} it was showed that every regular element in a $C^{\ast}$-algebra has Moore-Penrose inverse.
\smallbreak \smallbreak

\noindent {\it Drazin inverse.} In \cite{D}, the author introduced the concept of {\it pseudo-inverse} of an element $x$ in a ring (or a semigroup) $R$ as being an element $c\in R$ such that :
\begin{enumerate}\renewcommand{\theenumi}{\roman{enumi}}   \renewcommand{\labelenumi}{(\theenumi)}
\item $cx = xc$;
\item $x^m = x^{m+1}c$, for some positive integer $m$;
\item $c = c^2x$.
\end{enumerate}

It was showed in \cite{D} that if such an element $c$ exists, then it is uniquely determined and it commutes with every element of $R$ which commutes with $x$. Moreover, in the literature, the pseudo-inverse of an element $x\in R$ as defined in \cite{D} is called the {\it Drazin inverse} of $x$ and it is denoted by $x^D$. The smallest integer $m$ such that $x^m = x^{m+1}x^D$ is called the {\it Drazin index of $x$} and it is denoted by $i_D(x)$.
We observe that, in our context, if $u$ is $(e,e)$-invertible then $u$ has Drazin inverse and $i_D(u) = 1$. In fact,
\[uu^{-1} = e = u^{-1}u,\quad u^2u^{-1} = ue = u,\quad {(u^{-1})}^2u = u^{-1}e = u^{-1}.\]
A Drazin inverse with index $1$ was considered before by G. Azumaya, who called such an element {\it $\pi$-regular element}; see \cite{A}.
\smallbreak \smallbreak

\noindent {\it Corner-inner automorphism.} We recall from \cite{Cr}, that an automorphism $f$ of a ring $R$ is called {\it corner-inner} if there exists a nonzero idempotent $e\in R$ such that for
$e'=f^{-1}(e)$ there exist elements $u\in eRe'$ and $v\in e'Re$ such that $uv=e$, $vu=e'$ and $f(x)=uxv$, for all $x\in e'Re$. Note that, in this case, $u$
is $(e,e')$-invertible with inverse $v$.

\smallbreak \smallbreak
Let $C$ be a coalgebra and $A$ an algebra. Then it is well-known that $\Hom (C, A)$ is an algebra with the convolution product, that is,
if $\varphi, \psi \in \Hom(C, A)$ and $c\in C$, then $\varphi \ast \psi (c) = \varphi(c_1)\psi(c_2)$.

\begin{example} \label{example_invertibility}
Let $H$ be a weak Hopf algebra. Observe that $\varepsilon_t, \varepsilon_s$ are idempotents in
$\Hom(H,H)$. Consider the map $u(h) = h$, for all $h\in H$. Then $u$ is
$(\varepsilon_t,\varepsilon_s)$-invertible with $(\varepsilon_t,\varepsilon_s)$-inverse $u^{-1}(h)=S(h)$, where
$S$ is the antipode map. Indeed,
\begin{align*}
&u\ast u^{-1}(h) = h_1S(h_2) = \varepsilon_t(h),\quad u^{-1}\ast u(h) = S(h_1)h_2 = \varepsilon_s(h),&  \\
&\hspace{2.2cm}  u\ast f (h) = h_1\varepsilon_s(h_2) = h = u(h),& \\
&f\ast u^{-1}(h) = \varepsilon_s(h_1)S(h_2) = S(h_1)h_2S(h_3)=S(h) = u^{-1}(h),\quad  h\in H.&\
\end{align*}
\end{example}
\smallbreak

Let $C$ be a coalgebra. We recall that the coradical filtration of $C$ is the family $\{C_n\}_{n\geq 0}$ of subcoalgebras of $C$ satisfying:
\begin{enumerate}\renewcommand{\theenumi}{\roman{enumi}}   \renewcommand{\labelenumi}{(\theenumi)}
\item $C_0 \subset C_1 \subset \cdots \subset C_n \subset \cdots$;
\item $C = \cup_{n\geq 0} C_n$;
\item $C_0$ is the coradical of $C$ and
$C_n = \Delta^{-1}(C\otimes C_{n-1} + C_0\otimes C)$, for all $n\geq 1$.
\end{enumerate}

Let $A$ be an algebra. Note also that if $\gamma_1, \cdots, \gamma_n \in \Hom(C, A)$ are such that $\gamma_i(C_0) = 0$, $1\leq i \leq n$,  then
$\gamma_1 \ast \cdots \ast \gamma_n(C_{n-1}) = 0$. In particular, $\gamma^n(C_{n-1}) = 0$ for all $\gamma \in \Hom(C, A)$
such that $\gamma(C_0) = 0$. In this case, $\Theta := \sum_{n\geq 0} \gamma^n$ is a well defined map of $\Hom(C,A)$.

This last argumentation was used to prove Lemma 14 of \cite{Ta2}. Within the context of $(e,f)$-invertibility, we prove a similar result with almost the same argumentation as will be seen in the next result. First, we fix some notation. If $\psi \in \Hom (C,A)$,  then we will denote by $\psi_0$ its restriction to the coradical $C_0$ of $C$.

\smallbreak
\smallbreak

\begin{proposition}
Let $C$ be a coalgebra, $A$ an algebra, $e, f \in \Hom(C,A)$ idempotent elements and $\varphi \in \Hom(C, A)$ such that $e \ast \varphi = \varphi = \varphi \ast f$. Then $\varphi$ is $(e,f)$-invertible in $\Hom(C,A)$ if and only if $\varphi_{0}$ is $(e_0, f_0)$-invertible in $\Hom(C_0, A)$.
\end{proposition}

\begin{proof}
Clearly the $(e,f)$-invertibility of $\varphi$ implies the $(e_0, f_0)$-invertibility of $\varphi_0$. Conversely, denote by $\psi: C_0 \rightarrow A$ the $(e_0,f_0)$-inverse of $\varphi_0$ in $\Hom(C_0, A)$. Then,
\[\varphi_0 \ast \psi = e_0, \,\,  \psi \ast \varphi_0 = f_0, \,\, \psi \ast e_0 = \psi = f_0 \ast \psi \,\, \text{ and }  \,\, \varphi_0 \ast f_0 =
\varphi_0 = e_0 \ast \varphi_0.\]
Let $D$ be any complement of $C_0$ in $C$ as a vector space. Consider $\Psi \in \Hom(C, A)$ given by $\Psi_0 = \psi$ and
$\Psi|_{D} = 0$. We define the maps $\gamma_e, \gamma_f \in \Hom(C,A)$ by $\gamma_e = e - (\varphi \ast \Psi)$ and
$\gamma_f = f - (\Psi \ast \varphi)$. Since $\gamma_e$ and $\gamma_f$ vanish in $C_0$, it follows that $\Theta_e = \sum_{n\geq 0} \gamma_e^n$ and
$\Theta_f = \sum_{n\geq 0} \gamma_f^n$ are well defined maps. Because $e \ast \varphi = \varphi = \varphi \ast f$, it follows that $e\ast\gamma_e = \gamma_e$ and $f\ast \gamma_f = \gamma_f$. Using these observations we have the following equalities:
$$ (\varphi \ast \Psi) \ast \Theta_e = (e-\gamma_e)\ast\Theta_e = e,\,\,\,\,
\Theta_f \ast (\Psi \ast \varphi) = \Theta_f \ast (f-\gamma_f) = f.$$
Now we observe that
\begin{align*}
\Theta_f \ast \Psi \ast e  & = \Theta_f \ast \Psi \ast (\varphi \ast \Psi \ast \Theta_e) \\
                           & = (\Theta_f \ast \Psi\ast \varphi) \ast \Psi \ast \Theta_e \\
                           & =  (f \ast \Psi) \ast \Theta_e.
\end{align*}
Hence, if we call $\lambda := \Theta_f \ast \Psi \ast e  = f \ast \Psi \ast \Theta_e$, then it follows that $\lambda$ is the $(e, f)$-inverse of $\varphi$ in $\Hom(C, A)$.
\end{proof}

\section{Inner actions of weak Hopf algebras}
In this section we introduce the notion of inner  actions for weak Hopf algebras. We give sufficient conditions for an algebra to admit an inner  action of a weak Hopf algebra. We already noted that Sweedler introduced the notion of inner actions of Hopf algebras, but the interest in this kind of actions became stronger after \cite{BCM}.

We recall from \cite{BCM} the definition of inner actions of $H$ on $A$.

\begin{definition} \label{inneractionHA}
Let $H$ be a Hopf algebra and $A$ a left  $H$-module algebra with action denoted by $h\cdot a$. We say that this action is {\it inner}, if there exists an invertible  element $u\in \Hom(H,A)$ such that $h\cdot a = u(h_1)au^{-1}(h_2)$, for all $h\in H$ and $a\in A$.
\end{definition}

Let $C$ be a coalgebra, $A$ an algebra and $u\in \Hom(C,A)$ an invertible element. Following \cite{BCM}, we consider the bilinear map $t:C\times C \rightarrow A$, given by:
\begin{align}\label{bilinear}
 & t(x,y) = u^{-1}(y_1)u^{-1}(x_1)u(x_2y_2),& & x,y\in C.&
\end{align}

The next result is the Lemma 1.4 of \cite{BCM} which presents necessary and sufficient conditions for an action of  Hopf algebras to be inner.

\begin{proposition} \label{Lemma1.4BCM}
Let $H$ be a Hopf algebra, $A$ an algebra and $u\in \Hom(H,A)$ invertible and $h\cdot a := u(h_1)au^{-1}(h_2)$, for all $h\in H$, $a\in A$.
Then the following statements hold: for all $g,h\in H$ and $a,b\in A$,
\begin{enumerate}\renewcommand{\theenumi}{\roman{enumi}}   \renewcommand{\labelenumi}{(\theenumi)}
\item $u(1)$ is invertible in $A$ and $u(1)^{-1} = u^{-1}(1)$;
\item $h\cdot(ab) = (h_1\cdot a)(h_2\cdot b)$ and $h\cdot 1 = \varepsilon(1)1$;
\item $1 \cdot a = a$ if and only if $u(1) \in C(A)$;
\item $g\cdot (h \cdot a) = (gh)\cdot a$ if and only if $t(H\times H) \subseteq C(A)$.
\end{enumerate}
\end{proposition}

In what follows in this section we assume that $H$ is a weak Hopf algebra, $A$ is a $\F$-algebra, $e, f \in \Hom(H, A)$ are idempotent elements and $u$ is an $(e,f)$-invertible element of $\Hom(H,A)$. Now we propose a concept for inner actions of weak Hopf algebras.

\begin{definition} \label{inneractionWHA}
Let $A$ a left $H$-module algebra with action denoted by $h\cdot a$. We say that this action is {\it inner} if
\begin{align}\label{formula}
h\cdot a = u(h_1) a u^{-1}(h_2), \quad a\in A,\, h\in H.
\end{align}
In this case, we say that the inner action is {\em implemented} by $u$.
\end{definition}

\begin{remark}\label{idempotente}
	Let $A$ a left $H$-module algebra with inner action implemented by $u$. In this case, $e(h)=u(h_1)u^{-1}(h_2)=h\cdot 1$, for all $h\in H$. Moreover, we have $g\cdot e(h)=g\cdot(h\cdot 1)=gh\cdot 1=e(gh)$, for all $g,h\in H$.
\end{remark}

In the next result we investigate when the relation \eqref{formula} determines an structure of $H$-module algebra for $A$. Before that, we introduce more notation. Given $h\in H$, we fix $\varphi_{f,h}\in {\rm End}_{\F}(A)$ given by  $\varphi_{f,h}(a):=f(h_1)af(h_2)$, for all $a\in A$. We also fix $\lambda_{u,v}(h):=(u\otimes v)\circ \Delta (h)\in A\otimes A$, for all $h\in H$.

\begin{theorem} \label{analogousLemma1.4forWHA}
	The following statements hold:
	\begin{enumerate}\renewcommand{\theenumi}{\roman{enumi}}   \renewcommand{\labelenumi}{(\theenumi)}
		\item $h\cdot(ab) = (h_1\cdot a)(h_2\cdot b)$ if and only if $\varphi_{f,h}(ab)=\varphi_{f,h_1}(a)\varphi_{f,h_2}(b)$,  for all $h\in H$ and $a,b\in A$;\smallbreak
		\item $h\cdot 1 = \varepsilon_t(h)\cdot 1$ if and only if $e\circ \varepsilon_t = e$ if and only if $\kernel(\varepsilon_t)
		\subseteq \kernel(e)$, for all $h\in H$.\smallbreak
		\item Suppose $f(H) \subseteq C(A)$. Then, $g\cdot (h \cdot a) = (gh)\cdot a$ if and only if $g\cdot e(h)=e(gh)$ and $t(H\times H) \subseteq C(A)$, where $t$ is given in \eqref{bilinear}.\smallbreak
         \item If $u(H_s)\subseteq C(A)$ and $e(1)=1$ then $1\cdot a = a$, for all $a\in A$.\smallbreak
         \item If $\lambda_{u,v}(1)\in C_{A\otimes A}(u(H_s))$ and  $1\cdot a = a$ for all $a\in A$, then $e(1)=1$ and $u(H_s)\subset C(A)$.
\end{enumerate}
\end{theorem}

\begin{proof}
${\rm (i)}$ Let $h\in H$ and $a,b\in A$ and assume that  $\varphi_{f,h}(ab)=\varphi_{f,h_1}(a)\varphi_{f,h_2}(b)$. Then,
	\begin{align*}
	f(h_1)abf(h_2)&=f(h_1)af(h_2)f(h_3)bf(h_4) \\
	& =  f(h_1)af(h_2)bf(h_3).
	\end{align*}
	Hence,
	\begin{align*}
	h\cdot (ab) & =  u(h_1)abu^{-1}(h_2) \\
	& =  u(h_1)f(h_2)abf(h_3)u^{-1}(h_4) \\
	& =  u(h_1)f(h_2)af(h_3)bf(h_4)u^{-1}(h_5) \\
	& = u(h_1)af(h_2)bu^{-1}(h_3)    \\
	& = u(h_1)au^{-1}(h_2)u(h_3)bu^{-1}(h_4)    \\
	& =  (h_1\cdot a)(h_2\cdot b).
	\end{align*}
	
	Conversely,
	\begin{align*}
	& h\cdot (ab)  = (h \cdot a)(h\cdot b) \\
	\Rightarrow & \quad  u(h_1)abu^{-1}(h_2)=u(h_1)au^{-1}(h_2)u(h_3)bu^{-1}(h_4)\\
	\Rightarrow & \quad u^{-1}(h_1)u(h_2)abu^{-1}(h_3)u(h_4)\\
	&\,\,\,\, =\,\,u^{-1}(h_1)u(h_2)au^{-1}(h_3)u(h_4)bu^{-1}(h_5)u(h_6)\\
	\Rightarrow &\quad f(h_1)abf(h_2)=f(h_1)af(h_2)bf(h_3).\\
	\Rightarrow &\quad \varphi_h(ab)=\varphi_{h_1}(a)\varphi_{h_2}(b).
	\end{align*}
	
\noindent${\rm (ii)}$ Note that $g\cdot 1=e(g)$, for all $g\in H$. Thus, $h\cdot 1 = \varepsilon_t(h) \cdot 1$ if and only if
	$e(h) = e(\varepsilon_t(h))=e\circ\varepsilon_t(h)$, for all $h\in H$.
	The second equivalence follows because $\varepsilon_t$ is a projection. In fact, we have that $H = \kernel (\varepsilon_t) \oplus \imagem (\varepsilon_t)$ as a vector space. Therefore, $e(\varepsilon_t(h))= e(\varepsilon_t(h_0+h^{\prime}))=e(\varepsilon_t(h^{\prime}))= e(h^{\prime})$, for all $h\in H$ with $h = h_0+h^{\prime}$, $h_0 \in \kernel (\varepsilon_t)$ and $h^{\prime}\in \imagem (\varepsilon_t)$. If $\kernel (\varepsilon_t) \subseteq \kernel (e)$, then $e(h) = e(h_0) + e(h^{\prime}) = e(h^{\prime})$, and it follows that $e \circ \varepsilon_t = e$. The converse is immediate.
	
\noindent ${\rm (iii)}$ Suppose that $g\cdot (h \cdot a) = (gh)\cdot a$, for all $g,h\in H$ and $a\in A$. By Remark \ref{idempotente} it is enough to prove that $t(H\times H)\subset C(A)$. Let $g, h \in H$ and $a\in A$. Then,
\begin{align*}
& g\cdot (h \cdot a)  = gh \cdot a \\
\Rightarrow & \quad  u(g_1)u(h_1)au^{-1}(h_2)u^{-1}(g_2)\, =\, u(g_1h_1)au^{-1}(g_2h_2)\\
\Rightarrow & \quad u^{-1}(h_1)u^{-1}(g_1)u(g_2)u(h_2)au^{-1}(h_3)u^{-1}(g_3)u(g_4h_4) \\
&\,\,\,\, =\,\,u^{-1}(h_1) u^{-1}(g_1)u(g_2h_2)au^{-1}(g_3h_3)u(g_4h_4)\\
\Rightarrow &\quad u^{-1}(h_1)f(g_1)u(h_2)au^{-1}(h_3)u^{-1}(g_2)u(g_3h_4)\\
&\,\,\,\, =\,\, u^{-1}(h_1)u^{-1}(g_1)u(g_2h_2)af(g_3h_3) \\
\Rightarrow &\quad u^{-1}(h_1)u(h_2)au^{-1}(h_3)f(g_1)u^{-1}(g_2)u(g_3h_4)\\
&\,\,\,\, =\,\, u^{-1}(h_1)u^{-1}(g_1)u(g_2h_2)f(g_3h_3)a\\
\Rightarrow &\quad f(h_1)au^{-1}(h_2)u^{-1}(g_1)u(g_2h_3)\\
&\,\,\,\, =\,\,u^{-1}(h_1)u^{-1}(g_1)u(g_2h_2)a \\
\Rightarrow &\quad au^{-1}(h_1)u^{-1}(g_1)u(g_2h_2)\, =\,u^{-1}(h_1)u^{-1}(g_1)u(g_2h_2)a \\
\Rightarrow &\quad at(g,h)=t(g,h)a.
\end{align*}
\smallbreak
Conversely, suppose that $g\cdot e(h)=e(gh)$ and $at(g,h)=t(g,h)a$, for all $g,h\in H$ and $a\in A$. Then,
\begin{align*}
& at(g,h)=t(g,h)a \\
\Rightarrow &\quad au^{-1}(h_1)u^{-1}(g_1)u(g_2h_2)\,= \,u^{-1}(h_1)u^{-1}(g_1)u(g_2h_2)a \\
\Rightarrow &\quad u(g_1)u(h_1)au^{-1}(h_2)u^{-1}(g_2)u(g_3h_3)u^{-1}(g_4h_4)\\
&\,\,\,\, =\,\,u(g_1)u(h_1)u^{-1}(h_2)u^{-1}(g_2)u(g_3h_3)au^{-1}(g_4h_4) \\
\Rightarrow &\quad u(g_1)u(h_1)au^{-1}(h_2)u^{-1}(g_2)(g_3\cdot e(h_3))\\
&\,\,\,\, =\,\, (g_1\cdot e(h_1))u(g_2h_2)au^{-1}(g_3h_3) \\
\Rightarrow &\quad u(g_1)u(h_1)au^{-1}(h_2)u^{-1}(g_2)u(g_3)e(h_3)u^{-1}(g_4)\\
&\,\,\,\, =\,\, e(g_1h_1) u(g_2h_2)au^{-1}(g_3h_3) \\
\Rightarrow &\quad u(g_1)u(h_1)au^{-1}(h_2)f(g_2)e(h_3)u^{-1}(g_3)\\
&\,\,\,\, =\,\, u(g_1h_1)au^{-1}(g_2h_2) \\
\Rightarrow &\quad u(g_1)f(g_2)u(h_1)au^{-1}(h_2)e(h_3)u^{-1}(g_3)\\
&\,\,\,\, =\,\, u(g_1h_1)au^{-1}(g_2h_2) \\
\Rightarrow &\quad u(g_1)u(h_1)au^{-1}(h_2)u^{-1}(g_2)\, =\, u(g_1h_1)au^{-1}(g_2h_2) \\
\Rightarrow & \quad g\cdot(h\cdot a)=gh\cdot a.
\end{align*}	

\noindent ${\rm (iv)}$ Suppose that $u(H_s) \subseteq C(A)$, $e(1)=1$ and let $a\in A$. Then, $$1 \cdot a = u(1_1)au^{-1}(1_2) = au(1_1)u^{-1}(1_2) = a.$$

\noindent ${\rm (v)}$ it follows by Remark \ref{idempotente} that $e(1)=1$.  Given $x\in u(H_s)$ and $a\in A$ we have
\begin{align*}
xu(1_1)\otimes av(1_2)&=(1\otimes a)(xu(1_1)\otimes v(1_2))\\
                        &=(1\otimes a)(x\otimes 1)\lambda_{u,v}(1)\\
                        &=(1\otimes a)\lambda_{u,v}(1)(1\otimes x) \\
                        &=(1\otimes a)(u(1_1)\otimes v(1_2)x)\\
                        &=u(1_1)\otimes av(1_2)x.\
\end{align*}
Thus, $xa=xu(1_1)av(1_2)=u(1_1)av(1_2)x=ax$.
\end{proof}

We saw in Example \ref{example_invertibility} that $\varepsilon_t, \varepsilon_s$ are idempotents in
	$\Hom(H,H)$ and $u(h) = h$ is $(\varepsilon_t,\varepsilon_s)$-invertible with $(\varepsilon_t,\varepsilon_s)$-inverse $u^{-1}(h)=S(h)$, where
	$S$ is the antipode map. In this case, if we apply the relation \eqref{formula} then we obtain
	\begin{align}\label{case:adjoint}
	&\,\,\,\, h\cdot g=h_1gS(h_2),\quad\,\, h,g\in H.&
	\end{align}
Since $\Delta$ is multiplicative and $S$ is an anti-algebra morphism, it is immediate to check that $h\cdot(g\cdot l)=(hg)\cdot l$, for all $g,h,l\in H$. By Proposition 2.11 of \cite{BNS}, we have that $\lambda_{1,S}(1)=1_1\otimes S(1_2)$ is a separability idempotent element for the algebra  $H_s$. Hence $\lambda_{1,S}(1)\in C_{H\otimes H}(H_s)$. Moreover, $e(1)=\varepsilon_s(1)=1$. It follows, by Theorem \ref{analogousLemma1.4forWHA} (iv) and (v), the following result.

\begin{corollary}\label{cor:first} $H$ is an $H$-module via \eqref{case:adjoint} if and only if $H_s\subset C(H)$.
	\qed
\end{corollary}

Another consequence of Theorem \ref{analogousLemma1.4forWHA} is the following.

\begin{corollary}\label{cor:second}  If $f(H)\subset C(A)$ then
$h\cdot (ab)=(h_1\cdot a)(h_2\cdot b)$, for all $h\in H$ and $a,b\in A$.
\end{corollary}
\begin{proof} Let $h\in H$ and $a,b\in A$. Then
	\begin{align*}
	\varphi_{f,h_1}(a)\varphi_{f,h_2}(b)&=f(h_1)af(h_2)bf(h_3)=f(h_1)f(h_2)abf(h_3)\\
	&=f(h_1)abf(h_2)=\varphi_{f,h}(ab).
	\end{align*}
Hence, the result follows by Theorem \ref{analogousLemma1.4forWHA} (i).
\end{proof}

We end this subsection with another equivalent form of the definition of inner action of weak Hopf algebras, as in Lemma 1.14 of \cite{BCM}.

\begin{proposition} \label{secondforminneraction}
	Let $A$ be a left $H$-module algebra. Let $u \in \Hom(H, A)$ an $(e,f)$-invertible element such that $e(h)=h\cdot 1$, and $u(h_1)af(h_2)=u(h)a$, for all $a\in A$ and $h\in H$. Then, the action of $H$ on $A$ is an inner action implemented by $u$ if and only if $(h_1\cdot a)u(h_2) = u(h)a$, for all $h\in H$, $a\in A$.
\end{proposition}

\begin{proof}
	Suppose that $h\cdot a = u(h_1)au^{-1}(h_2)$, for all $a\in A$, $h\in H$. Then,
	\begin{align*}
	(h_1\cdot a)u(h_2)  &=  (u(h_1)au^{-1}(h_2))u(h_3) = u(h_1)af(h_2) =  u(h)a.\
	\end{align*}
	Conversely, for all $a\in A$, $h\in H$, we have that
	\begin{align*}
	h\cdot a &=  (h_1 \cdot a)(h_2\cdot 1) = (h_1\cdot a)e(h_2) \\
	&=  (h_1\cdot a)u(h_2)u^{-1}(h_3) = u(h_1)a u^{-1}(h_2).\
	\end{align*}
\end{proof}

\section{Adjoint actions and quantum commutativity}
The notion of quantum commutative weak Hopf algebra was given in \cite{AVRC} for any weak Hopf algebra in a strict symmetric monoidal category. Denote by {\bf vec}$_{\F}$ the category of finite dimensional $\F$-vector spaces with the usual symmetry. Then, a weak Hopf algebra in {\bf vec}$_{\F}$ means an ordinary weak Hopf algebra. Moreover, in this case, a weak Hopf algebra $H$ is {\it quantum commutative} if
\begin{align}\label{def:quantum_commutative}
h_1g\varepsilon_s(h_2)=hg,\quad \quad\text{ for all } h,g\in H.
\end{align}

Now we give a characterization of the notion of quantum commutative weak Hopf algebra in {\bf vec}$_{\F}$.

\begin{proposition} \label{quantum_commutative}
	Let $H$ be a weak Hopf algebra. Then $H$ is quantum commutative if and only if $H_s\subseteq C(H)$.
\end{proposition}

\begin{proof}
Assume that $H_s\subseteq C(H)$ and take $g,h\in H$. Then
\[g_1h\varepsilon_s(g_2)=g_1\varepsilon_s(g_2)h=gh.\]
\noindent Conversely, if $H$ is quantum commutative then $1_1h\varepsilon_s(1_2)=h$, for all $h\in H$. By Proposition 2.11 of \cite{BNS}, $q=1_1\otimes S(1_2)$ is a separability idempotent element for the algebra  $H_s$. Using Lemma 2.9 of \cite{BNS}, we have that \[q=1_1\otimes S(1_2)=1_1\otimes S(\varepsilon_t(1_2))=1_1\otimes \varepsilon_s(\varepsilon_t(1_2))=1_1\otimes \varepsilon_s(1_2).\] Hence, $x1_1\otimes \varepsilon_s(1_2)=1_1\otimes \varepsilon_s(1_2)x$,  for all $x\in H_s$. Multiplying in the left side by $h\otimes 1$, $h\in H$, and applying the multiplication map, it follows that
\[xh=x1_1h\varepsilon_s(1_2)=1_1h\varepsilon_s(1_2)x=hx.\]
\end{proof}

The next result is the Corollary 2.16 of \cite{AVRC}, which can be obtained using developed results in this paper.

\begin{corollary}\label{cor:espanhois}
	$H$ is an $H$-module algebra via \eqref{case:adjoint} if and only if $H$ is quantum commutative.
\end{corollary}
\begin{proof} Suppose that $H$ is an $H$-module algebra. By Corollary \ref{cor:first} and Proposition \ref{quantum_commutative}, we have that $H$ is quantum commutative. Conversely, assume that $H$ is quantum commutative. Then, by Proposition \ref{quantum_commutative}, $H_s\subset C(H)$. Consequently, it follows from  Corollary \ref{cor:first} that $H$ is an $H$-module. Since $\varepsilon_t$ is an idempotent map, we obtain from Theorem \ref{analogousLemma1.4forWHA} (ii) that $h\cdot 1=\varepsilon_t(h)\cdot 1$, for all $h\in H$. By Corollary \ref{cor:second}, $h\cdot (ab)=(h_1\cdot a)(h_2\cdot b)$, for all $h\in H$ and $a,b\in A$. Therefore, $H$ is an $H$-module algebra.
\end{proof}

\begin{remark}\label{rem:2.2}
	Let $H$ be a weak Hopf algebra, $A$ a left $H$-module algebra, $h,g\in H$ and $x\in A$.
 It is clear that the
	linear maps $A\longrightarrow A\#H,$ $a\mapsto a\#1$  and $H\longrightarrow A\#H,$ $h\mapsto h\#1$, are
	injective morphism of algebras. Furthermore,
	\begin{align*}
	h \cdot (x\#1_H) & = h_1 \cdot x \# h_2 S(h_3)=(h_1 \cdot x)\cdot \vet(h_2) \#1_H\\
	&\stackrel{\eqref{actiondimitri}}=(h_1 \cdot x)(\vet(h_2)\cdot 1_A)\#1_H\\
	&=(h_1 \cdot x)(h_2\cdot 1_A)\#1_H\\
	&=(h\cdot x)\#1_H.
	\end{align*}
\end{remark}	

Given $H$ a weak Hopf algebra and $A$ a left $H$-module algebra, we define:
\begin{align*}
& h\cdot (a\#g) = (h_1\cdot a)\# h_2gS(h_3), \text{ for all } h,g\in H,\,\,a\in A.&
\end{align*}
Consider $e, f, u, v \in \Hom(H,A\#H)$ given by:
	\begin{align*}
	&e(h) = (h\cdot 1) \# 1,& &f(h) = 1 \# \varepsilon_s(h),& &u(h) = 1\# h,&  &v(h) = 1 \# S(h).&
	\end{align*}
It is easy to check that $v$ is the $(e,f)$-inverse of $u$. Moreover, for all $h,g\in H$ and $a\in A$, we have:
\begin{align}\label{action_smash}
&h\cdot (a\#g)= h_1 \cdot a \# h_2 g S(h_3)=u(h_1) (a \# g) v(h_2).&
\end{align}

In the next result we present necessary and sufficient conditions in order that the equation \eqref{action_smash} provides an $H$-module structure on $A\#H$.

\begin{proposition}\label{some_equivalences} Let $H$ be a weak Hopf algebra and $A$ a left $H$-module algebra. The following conditions are equivalent:
\begin{enumerate}\renewcommand{\theenumi}{\roman{enumi}}   \renewcommand{\labelenumi}{(\theenumi)}
	\item $A\#H$ is a left $H$-module algebra via the inner action given in \eqref{action_smash}.
	\item $1\#1_1gS(1_2)=1\#g$, for all $g\in H$;
	\item $1\#h_1g\varepsilon_s(h_2)=1\#hg$, for all $h,g\in H$;
	\item $1\#H_s\subset C(A\#H)$;
	\item $H$ is quantum commutative.
\end{enumerate}
\end{proposition}

\begin{proof} (i) $\Rightarrow$ (ii) Given $h\in H$, we have
	\begin{align*}
	1\#h&=1\cdot (1\#h)=1_1\cdot 1\#1_2hS(1_3)\\
	    &=\varepsilon_t(1_1)\cdot 1\#1_2hS(1_3)=1_1\cdot \varepsilon_t(1_1)\#1_2hS(1_3)\\
	    &=1\#\varepsilon_t(1_1)1_2hS(1_3)=1\#1_1hS(1_2).\
	\end{align*}

\noindent (ii) $\Rightarrow$ (iii) Let $g,h\in H$. Then,
	\begin{align*}
    1\#hg&=(1\#h)(1\#g)=(1\#h)(1\#1_1gS(1_2))\\
         &=h_1\cdot 1\#h_21_1gS(1_2)=\varepsilon_t(h_1)\cdot 1\#h_21_1gS(1_2)\\
         &=1\#h1_1gS(1_2)=1\#h_1gS(h_2)h_3=1\#h_1g\varepsilon_s(h_2).\
	\end{align*}
We used \cite[2.30c]{BNS} for the penultimate equality in the equation above.
\smallbreak

\noindent (iii) $\Rightarrow$ (iv) By Remark \ref{rem:2.2}, one can see that under our hypothesis $H$ is quantum commutative. It follows by Proposition \ref{quantum_commutative} that $H_s\subset C(H)$.
Thus, $(1\#x)(a\#g)=(a\#xg)=(a\#gx)=(a\#g)(1\#x)$, for all $x\in H_s$, $g\in H$ and $a\in A$.
\smallbreak

\noindent (iv) $\Rightarrow$ (i) Let $g, h, l \in H$ and $x,y \in A$. It is clear that $g \cdot (h
\cdot (x \# l))=(gh)\cdot (x\# l)$. Since $H_s\subset C(H)$,
\begin{align*}
1_H \cdot (x \# h) & =  1_1 \cdot x \# 1_2 h S(1_3) = 1_1 \cdot x \# 1_2 S(1_3) h  \\
& = 1_1 \cdot x \# \varepsilon_t(1_2) h=1_1 \cdot x \# 1_2 h = (1_1 \cdot x)\cdot 1_2 \# h\\
& \overset{\eqref{actiondimitri}}=  (1_1 \cdot x) (1_2 \cdot 1_A) \# h =  x \# h,
\end{align*}
\begin{align*}
h \cdot [(x \# g)(y \# l)] & =  h_1 \cdot (x (g_1 \cdot y)) \# h_2 g_2 l S(h_3) \\
& =  h_1 \cdot (x (g_1 \cdot y)) \# h_2 \varepsilon_s (h_3) g_2 l S(h_4) \\
& =  h_1 \cdot (x (g_1 1_1 \cdot y)) \# h_2 g_2 1_2 \varepsilon_s (h_3) l S(h_4) \\
& =  h_1 \cdot (x (g_1 1_1 \cdot y)) \# h_2 g_2 S(h_3) h_4 1_2 l S(h_5) \\
& =  h_1 \cdot (x (g_1 1_1 \cdot y)) \# h_2 g_2 S(h_3) \varepsilon (h_4 1_2) h_5 l S(h_6) \\
& =  h_1 \cdot (x (g_1 1_1 \varepsilon (h_4 1_2) \cdot y)) \#  h_2 g_2 S(h_3) h_5 l S(h_6) \\
& =  h_1 \cdot (x (g_1 \varepsilon_s (h_4) \cdot y)) \# h_2 g_2 S(h_3) h_5 l S(h_6) \\
& =  (h_1 \cdot x) (h_2 \cdot (g_1 \varepsilon_s (h_5) \cdot y)) \# h_3 g_2 S(h_4) h_6 l S(h_7) \\
& =  (h_1 \cdot x)(h_2 g_1 S(h_5) \cdot (h_6 \cdot y)) \# h_3 g_2 S(h_4) h_7 l S(h_8) \\
& =  (h_1 \cdot x \# h_2 g S(h_3)) ( h_4 \cdot y \# h_5 l S(h_6)) \\
& =  (h_1 \cdot (x \# g))(h_2 \cdot (y \# l)).
\end{align*}
Also, using Remark \ref{rem:2.2} (i), follows that,
\[h \cdot (1_A \# 1_H)  = (h \cdot 1_A) \# 1_H = (\varepsilon_t(h) \cdot 1_A) \# 1_H = \varepsilon_t(h) \cdot (1_A \# 1_H).\]
\noindent (iii) $\Leftrightarrow$ (v)  follows by Remark \ref{rem:2.2}.
\end{proof}

We end this section with the particular case of groupoid algebras. Let $\G$ be a groupoid, that is,
a small category such that each morphism is invertible. Given $g\in \G$, we denote
\[ s(g):=g^{-1}g,\,\,\, t(g):=gg^{-1},\,\,\,\G_0:=\{s(g)\,:\,g\in \G\}.\]
Given $\alpha\in \G_0$,  we consider the isotropy group $\G_\alpha:=\{g\in \G\,:\,s(g)=t(g)=\alpha\}$ associated to $\alpha$. Consider the $\F$-vector space $\F\G$ with
bases $\G$. If $\G$ is finite, then $\F\G$ is a weak Hopf algebra with the following structures:
\begin{align*}
&m(g\otimes h)=\left\{\begin{array}{cc}
gh,&\text{ if } s(g)=t(h) \\
0,& \text{otherwise}
\end{array}\right.,& &\mu(1)=\sum_{\alpha\in \G_o} \alpha,& \\
&\Delta(g)=g\otimes g ,\quad \varepsilon(g)=1\,\,\text{ and }\,\, S(g)=g^{-1},& &\text{ for all } g\in \G.&
\end{align*}

Let $\G$ be a finite groupoid, $A$ a left $\F\G$-module algebra and consider $e(g)=1\#t(g)$ and $f(g)=1\#s(g)$, for all $g\in \G$. Then the element $u\in \Hom(\F\G, A\#\F\G)$ given by $u(h) = 1\# h$ is $(e,f)$-invertible with inverse $u^{-1}(h) = 1\# h^{-1}$, for all $h\in \G$. In this case, we can rewrite the relation \eqref{action_smash} as
\begin{align}\label{weakinneraction}
&h \cdot (x \# g):=u(h_1)(x\#g)u^{-1}(h_2)=h\cdot x\# h g h^{-1},& &g,h\in \G,\,\,x\in A.&
\end{align}

 \begin{corollary} \label{example_groupoid}
	Let $\G$ be a finite groupoid and $A$ a left $\F\G$-module algebra.
	Then,  \eqref{weakinneraction} determines an action of $\F \G$ on
	$A\#\F\G$ if and only if $\G=\coprod_{\alpha\in
		\G_0}\G_\alpha$ (that is, $\G$ is a disjoint union of its isotropy groups).
\end{corollary}

\begin{proof}
It follows from \cite[Section 2.5]{NikVai} that $(\F\G)_s=(\F\G)_t=\oplus_{\alpha\in \G_0}\F \alpha$, which implies that$(\F\G)_s\subset C(\F\G)$ if and only if $\G=\coprod_{\alpha\in \G_0}\G_\alpha$. Hence, the result follows by Proposition \ref{quantum_commutative} and Proposition \ref{some_equivalences}.
\end{proof}

\subsection*{Acknowledgments}
The authors would like to thank  V. Rodrigues for fruitful discussions and A. Paques for his comments and suggestions.

\end{document}